\documentclass[a4paper,12pt]{amsart}
\usepackage{graphicx}
\usepackage[psamsfonts]{amssymb}
\usepackage{amscd,pifont}
\usepackage{cite}
\usepackage[usenames]{color}
\usepackage{amssymb,amsmath,amsthm}
\usepackage[colorlinks,linkcolor=red,anchorcolor=blue,citecolor=blue]{hyperref}

\makeatletter
\@namedef{subjclassname@2020}{\textup{2020} Mathematics Subject Classification}
\makeatother

\newtheorem{theorem}{Theorem}[section]
\newtheorem{lemma}[theorem]{Lemma}
\newtheorem{proposition}[theorem]{Proposition}
\newtheorem{corollary}[theorem]{Corollary}
\newtheorem{definition}[theorem]{Definition}

\def\N{\mathbb{N}}
\def\Z{\mathbb{Z}}
\def\Q{\mathbb{Q}}

\def\Zp{\mathbb{Z}_p}
\def\Qp{\mathbb{Q}_p}

\newcommand{\abs}[1]{|#1|}
\newcommand{\abspp}[1]{|#1|_p}
\newcommand{\norm}[1]{\|#1\|}

\newcommand{\supp}{\text{supp}}

\begin{document}

\title{On  product spectral sets and functional tiles in $\Q_p^d$}

\author{Mamateli Kadir}
\address{School of Mathematics and Statistics \& Research Center of Modern Mathematics and Applications,
Kashi University, Kashi, 844006, P. R. China}

\email{mamatili880@163.com}


\begin{abstract}
This paper studies product spectral sets and functional tiles in the $p$-adic spaces $\mathbb{Q}_p^d$
within the framework of the {\bf product spectral set conjecture}. We establish  a stability result for functional
tiles under weak convergence of tiling complements, characterize product spectral pairs
with product spectra, and fully resolve the conjecture for cylindric sets by proving that
spectrality of a cylindric set $\Omega=B_\gamma(a)\times\Theta$ is equivalent to spectrality of $\Theta$.
\end{abstract}



\maketitle

\section{Introduction}

Let $G$ be a locally compact Abelian group and $\widehat{G}$  be its dual group.
A Borel set  $\Omega\subset G$ of positive and finite Haar measure is called a {\bf spectral set}
if there exists a set $\Lambda \subset \widehat G$ of continuous characters of $G$ which forms
an orthogonal basis of the space $L^2(\Omega)$ of square Haar-integrable functions.
In this case, $\Lambda$ is called a {\bf spectrum} of $\Omega$ and $(\Omega,\Lambda)$ is called
a {\bf spectral pair}.
We say that a set $\Omega$ {\bf tiles} $G$ by translation if there exists a set $T\subset G$
such that the family of sets $\{\Omega+t\}_{t\in T}$ constitutes a partition of $G$ up to measure zero.
Such a set $T$  is called a {\bf tiling complement} or a  {\bf tiling set} of $\Omega$ and
$(\Omega,T)$ is called a {\bf tiling pair}.

The study of spectral sets and tiling properties, originating from Fuglede's seminal work \cite{Fuglede1974},
investigates the profound connection between the geometry of a set and the harmonic analysis it supports.
This connection arises naturally in several areas of mathematics and mathematical physics, including the
theory of commuting partial differential operators, the structure of quasicrystals, and  wavelet constructions.

A central question, known as \textbf{Fuglede's conjecture} or \textbf{spectral set conjecture}, proposes that a Borel set $\Omega\subset\mathbb{R}^d$ of positive and finite Lebesgue measure is
  a \textbf{spectral set} if and only if it is a \textbf{translational tile}.

Significant progress has been made in understanding this conjecture in Euclidean space $\mathbb{R}^d$.
For instance, it is known that the Euclidean balls \cite{Fuglede2001} or any convex domains with a smooth
boundary \cite{Iosevich-Katz-Tao} or non-symmetric convex domains \cite{K2000}, are not spectral. Clearly,
they are not translational tiles either. A complete classification has been achieved for convex domains in $\mathbb{R}^d$,
establishing that a convex domain is spectral if and only if it is a translational tile \cite{Iosevich-Katz-Tao2008, LM21}.
However, the general conjecture was disproved by Tao \cite{Tao2004}, who showed the ``spectral $\Rightarrow$ tile''
direction fails in dimensions $d \geq 5$. Subsequent work provided counterexamples for both directions in dimensions $d \geq 3$ \cite{FMM2006,KM2006,KM2006a,Matolcsi2005}.
The conjecture remains open in dimensions $d =1, 2$.

The \textbf{generalized  Fuglede's conjecture} for a locally compact Abelian group $G$
states that a Borel set $\Omega\subset G$ of positive and finite Haar measure is a \textbf{spectral set}
if and only if it is a \textbf{translational tile}.

This generalized formulation has attracted considerable attention and has been investigated in finite
Abelian groups \cite{AIMP2015, IMP17, KMSV20, KMSV22, KS21, MK17, S19, S20, Somlai21},
as well as in non-Archimedean locally compact fields
\cite{Fan,Fan-Fan-Shi,Fan-Fan-Liao-Shi,Kadir}.

In the $p$-adic field $\Qp$, the conjecture has been completely resolved for one-dimensional sets.
A Borel set $\Omega \subset \mathbb{Q}_p$ is spectral \textit{if and only if} it is a translational
tile \textit{if and only if}  it is an almost compact open set with a specific $p$-homogeneous structure
 \cite{Fan-Fan-Shi,Fan-Fan-Liao-Shi}.
There are some counterexamples in higher-dimensional $p$-adic spaces $\mathbb{Q}_p^{d}$ for $d \geq 3$,
but the conjecture remains open in $\Q_p^2$.

Fuglede's original motivation came from commuting self-adjoint operators in quantum mechanics. Spectral sets
in $p$-adic spaces may find interpretations in $p$-adic quantum models or in the study of vibrational modes on ultrametric lattices.

An important problem connected to the  spectral set conjecture is whether the spectrality (resp. tiling property)
of a product set in a product group implies the spectrality (resp. tiling property) of the factors.
This leads to the following conjecture. Let $ G_1$ and $G_2$ be two locally compact Abelian groups,
and $\Omega_1\subset G_1$ and $ \Omega_2\subset G_2$ be two bounded, measurable sets.

\textbf{Product spectral set conjecture:}  A product set $\Omega=\Omega_1\times\Omega_2$ is a spectral set (resp. a tile) if and only if
$\Omega_1$ and $\Omega_2$ are both spectral sets (resp. tiles).

The tiling part of this conjecture is  relatively simple and has been settled in $\mathbb{R}^d  $\cite{Kolountzakis2016},
and in $\mathbb{Q}_p^d$ \cite{K2019}.
Research has also been focused on the spectral part in the Euclidean setting.
The conjecture has been confirmed when $\Omega_1\subset\mathbb{R}$ is an interval \cite{Greenfeld-Lev2016} or
union of two intervals \cite{Kolountzakis2016}, when $\Omega_1\subset\mathbb{R}^2$ is a convex
polygon \cite{GL2020}, and when $\Omega_1$ is a classical Sierpinski self-affine tile  and $\Omega_2$
is a Lebesgue measurable set of measure one \cite{CLZ}. However, Somlai \cite{Somlai}
recently disproved this conjecture by constructing a counterexample showing that
there exist bounded measurable sets $\Omega_1$ and $\Omega_2$ in $\mathbb{R}^3$ such that
$\Omega_1$ is not spectral, yet $\Omega_1\times\Omega_2$ is spectral in $\mathbb{R}^6$.
Although the spectral set conjecture has been completely resolved in one-dimensional case $\Qp$,
no one has  considered the product spectral set conjecture in $\Qp^d$, and we make the first step in this direction.

We could generalize the  notion of tiles to functions. Let $f\in L^1(\Q_p^d)$ be a non-negative function,
and  $\Lambda\subset \Q_p^d$ be a discrete set. We say that $f$ tiles $\Q_p^d$ by translation
with $\Lambda\subset\Q_p^d$, if we have
$$
\sum_{\lambda\in\Lambda}f(x-\lambda)=1, \quad a.e.\ x\in\Q_p^d.
$$
In this case, we say that  $(f, \Lambda)$ is a tiling pair.
If $f=1_\Omega$ is the indicator function of a Borel set $\Omega\subset\Q_p^d$, then
the functional tiles are just the translational tiles.

The study of translational tiles in $\Qp^d$ contributes to the general theory of tiles in metric measure
spaces, with potential links to questions about covering, packing, and discrepancy in ultra-metric geometries.

In this paper, we focus on the spectral part of the {\bf product spectral set conjecture}  and functional tiles
in the  non-Archimedean locally compact space $\mathbb{Q}_p^d$. Our main contributions are threefold:\\
\indent 1. {\bf Stability of functional tiles under weak convergence}.

A set \(\Lambda \subset \mathbb{Q}_p^d\) is called \textbf{uniformly discrete} if there exists a \(\delta >0\)
such that for any two distinct points \(\lambda, \lambda' \in \Lambda\), we have
\[
|\lambda - \lambda'|_p \ge \delta.
\]
The largest such \(\delta\) is called the \textbf{separation constant}, and is denoted by \(\delta(\Lambda)\).

Let $\Lambda_n$ be a sequence of uniformly discrete sets in $\Qp^d$, with separation constants
$\delta(\Lambda_n)\ge \delta$. The sequence $\Lambda_n$ is said to converge weakly to a set $\Lambda$
if for every $\eta\in\N$ and every $\gamma\in\Z$ there exists an $N$ such that
$$
\Lambda_n\cap B_{\gamma}^d\subset \Lambda+ B_{-\eta}^d \quad and \quad \Lambda\cap B_{\gamma}^d\subset \Lambda_n+ B_{-\eta}^d
$$
for all $n\geq N$, where $B_{\gamma}^d$ denotes the open ball of radius $p^{\gamma}$ centered at the origin.

In this case the weak limit $\Lambda$ is also uniformly discrete with separation
constant $\delta(\Lambda)\ge \delta$.
The following theorem confirms that the tiling property of a function is invariant under weak convergence.
\begin{theorem}\label{weak convergence}
Let $\Lambda_n$ be a sequence of uniformly discrete sets  with separation constants
$\delta(\Lambda_n)\ge \delta$, weakly convergent to $\Lambda$. If $f\in L^1(\Q_p^d)$ is non-negative
and $(f, \Lambda_n)$ is a tiling pair for all $n\in\N$, then $(f, \Lambda)$ is also a tiling pair.
\end{theorem}

This theorem provides a topological tool for studying deformations and
approximations of tiling structures.

\indent 2. {\bf Product spectrality with product spectra}.

Our first main result on the product spectral set conjecture gives a positive answer when spectra have product forms.

\begin{theorem}\label{product}
Let $\Omega_1\subset\Q_p^{d_1}$, $\Omega_2\subset\Q_p^{d_2}$ be bounded  measurable sets of positive measure,
and let $\Lambda_1\subset\Q_p^{d_1}$, $\Lambda_2\subset\Q_p^{d_2}$ be  discrete sets.
Denote $\Lambda=\Lambda_1\times\Lambda_2$ and $\Omega=\Omega_1\times\Omega_2$.
Then $(\Omega,\Lambda)$ is a spectral pair in $\Q_p^{d_1+d_2}$ if and only if $(\Omega_1,\Lambda_1)$
is a spectral pair in $\Q_p^{d_1}$, and $(\Omega_2,\Lambda_2)$ is a spectral pair in $\Q_p^{d_2}$.
\end{theorem}

This result extends earlier Euclidean results \cite{Jorgensen-Pedersen1999} to the $p$-adic
setting and clarifies the structural role of product spectra.

\indent 3. {\bf Spectrality of the cylindric sets}.

We then focus on a specific class of sets in $\Q_p^d$. A Borel set
 $\Omega\subset\Q_p^d \ (d\geq2)$  of positive and finite Haar measure is called
a \textbf{cylindric} set, if it has the form
 \begin{equation}
 \Omega=B_\gamma(a)\times\Theta,
\end{equation}
where $B_\gamma(a)\subset\Q_p$ is a ball centered at $a$ of radius $p^\gamma$,
and $\Theta$ is a Borel set in $\Q_p^{d-1}$. It is clear that balls $B_\gamma(a)$ in
$\Q_p$ are spectral sets.
For such cylindric sets, we resolve the product spectral set conjecture completely.
\begin{theorem}\label{main}
Let $\Omega=B_\gamma(a)\times\Theta$ be a cylindric set in $\Qp^d$. Then $\Omega$ is a {\em spectral set} in $\Qp^d$
if and only if $\Theta$ is a {\em spectral set} in $\Qp^{d-1}$.
\end{theorem}

This result reduces the spectrality of a cylindric set in $\mathbb{Q}_{p}^{d}$ to the spectrality of a set
in $\Qp^{d-1}$, providing a powerful tool for constructing lower dimensional spectral sets from
known higher dimensional ones, and vice versa.

The`` if " direction follows directly from Theorem \ref{product},
but the ``only if"  direction requires a careful analysis of the Fourier transform of $p$-adic
balls and a construction of spectra constrained to a quasi-lattice structure. This result reduces
the spectrality question in dimension $d$ to one in dimension $d-1$, offering a new method for
constructing spectral sets inductively.

Our approach combines tools from $p$-adic harmonic analysis, weak convergence methods, and geometric
properties of ultrametric spaces. The results contribute not only to the Fuglede's conjecture in
non-Archimedean contexts, but also to the broader study of spectral synthesis, tiling theory,
and basis properties of exponential systems in locally compact Abelian groups. They may also have
implications for the construction of orthogonal bases in $p$-adic wavelet theory and for understanding
tiling and spectral properties in adelic or product-of-local-fields settings.

The paper is organized as follows. In Section \ref{prelim}, we recall essential preliminaries on $p$-adic analysis,
Fourier transforms, and the concepts of tiling and spectrality.
Section \ref{main results} is devoted to the proofs of our main theorems.

\section{Preliminaries}\label{prelim}
This section presents essential foundations for proving the paper's main theorems,
including the algebraic-topological properties of the $p$-adic field $\Qp$,
$p$-adic integration theory, Fourier transform of functions in $L^1(\Qp^d)$,
and key criteria for spectral sets and tiling sets.

\subsection{The field $\Qp$ of $p$-adic numbers}

We start with a brief review  of $p$-adic numbers. Let \(\mathbb{Q}\) denote the field of rational numbers,
and let \(p \geq 2\) be a prime number. Any non-zero element \(x \in \mathbb{Q}\) can be written as \(x = p^v \cdot \frac{a}{b}\),
where \(v, a, b \in \mathbb{Z}\), and \(\gcd(p, a) = \gcd(p, b) = 1\).
By the unique factorization property of integers, the exponent \(v\) is uniquely determined
by \(x\). We define the \(p\) adic valuation \(v_p(x) = v\) for \(x \neq 0\) and
\(v_p(0) = +\infty\), along with the \(p\)-adic absolute value \(|x|_p = p^{-v_p(x)}\).
This absolute value \(|\cdot|_p \) is non-Archimedean, meaning it satisfies three properties:\\

\indent (i)  \  $|x|_p\ge 0$ with equality only for $x=0$; \\
\indent (ii) \ $|x y|_p=|x|_p |y|_p$;\\
\indent (iii) $|x+y|_p\le \max\{ |x|_p, |y|_p\}$.\\

The field \(\Qp\) of \(p\)-adic numbers is the completion of \(\mathbb{Q}\) with respect to the
\(p\)-adic absolute value \(|\cdot|_p\). Every element \(x \in \Qp\) admits a unique expansion as
\[
x = \sum_{n=v}^{\infty} a_n p^n \quad \left(v \in \mathbb{Z}, \, a_n \in \{0, 1, \dots, p-1\}, \, \text{and } a_v \neq 0\right),
\]
where \(v(x) := v\) is the \(p\)-adic valuation of \(x\), and \(|x|_p= p^{-v}\).
The ring of \(p\)-adic integers, denoted by \(\Zp\), is the unit ball centered at the origin
\begin{align}\label{element}
\Zp = \Big\{x \in \Qp : x = \sum_{n=0}^{\infty} a_n p^n\Big\}.
\end{align}
The fractional part of \(x \in \Qp\) is defined as
\[
\{x\} = \sum_{n=v}^{-1} a_n p^n
\] when \(v<0\), and  \(\{x\}=0\) otherwise.

Since the \(p\)-adic absolute value only takes values in \(\{p^\gamma : \gamma \in \mathbb{Z}\} \cup \{0\}\), we only need to consider balls of radius \(r = p^\gamma\). For \(a \in \Qp\), we denote
the closed ball of radius \(p^\gamma\) centered at \(a\) by\[B_\gamma(a) = \big\{x \in \Qp : \abspp{x-a} \leq p^\gamma\big\},\]
and sphere of radius \(p^\gamma\) centered at \(a\) by \[S_\gamma(a) = \big\{x \in \Qp : \abspp{x-a} = p^\gamma\big\}.\]

We abbreviate \(B_\gamma := B_\gamma(0)\) and \(S_\gamma := S_\gamma(0)\) for simplicity.
As \(\Qp\) is an ultrametric space, \(B_\gamma\) forms an additive subgroup of \(\Qp\).
The set \(B_\gamma(a)=a+B_\gamma\) is a coset of the subgroup \(B_\gamma\).
It is itself a subgroup only when \(a \in B_\gamma\). Additionally, we have
\[
S_\gamma(a) = B_\gamma(a) \setminus B_{\gamma-1}(a), \quad \bigcup_{\gamma \in \mathbb{Z}} B_\gamma(a) = \bigcup_{\gamma \in \mathbb{Z}} S_\gamma(a) = \Qp, \quad \bigcap_{\gamma \in \mathbb{Z}} B_\gamma(a) = \{a\}.
\]

Notation:
\begin{itemize}
    \item \(B_0 = \Zp\) (the ring of \(p\)-adic integers);
    \item \(S_0 = \Zp^\times\) (the multiplicative group of invertible elements in \(\Zp\));
    \item \(B_{-1} = p\Zp\) (the unique maximal ideal of \(\Zp\)).
\end{itemize}

By the ultrametric property, every point in \(B_\gamma(a)\) can serve as its center, and both \(B_\gamma(a)\) and \(S_\gamma(a)\) are clopen sets in \(\Qp\). For any two balls in \(\Qp\), their intersection is non-empty if and only if one is contained in the other, i.e.
\[
B_\gamma(a) \cap B_{\gamma'}(b) \neq \emptyset \iff B_\gamma(a) \subset B_{\gamma'}(b) \text{ or } B_{\gamma'}(b) \subset B_\gamma(a).
\]
In other words, two balls in \(\Qp\) are either disjoint or nested.

When we consider the geometric structure of $\mathbb{Q}_p$ and its subsets, we always keep in our minds two models of $\mathbb{Q}_p$,
the ball model and the tree model . See Figure \ref{Fig:ball model} and Figure \ref{Fig:tree model} for ball model and tree model of the  geometric structure of $\mathbb{Q}_3$, respectively.

\begin{figure}[h!]
	\centering
	\includegraphics[width=0.4\textwidth]{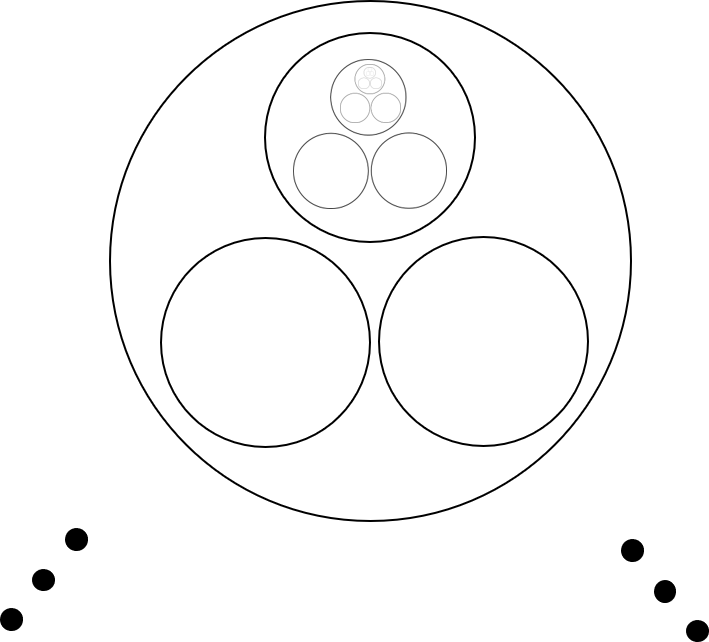}\\
	\caption{Ball model of $\Q_3$.}
	\label{Fig:ball model}
\end{figure}

\begin{figure}[h!]
	\centering
	\includegraphics[width=0.8\textwidth]{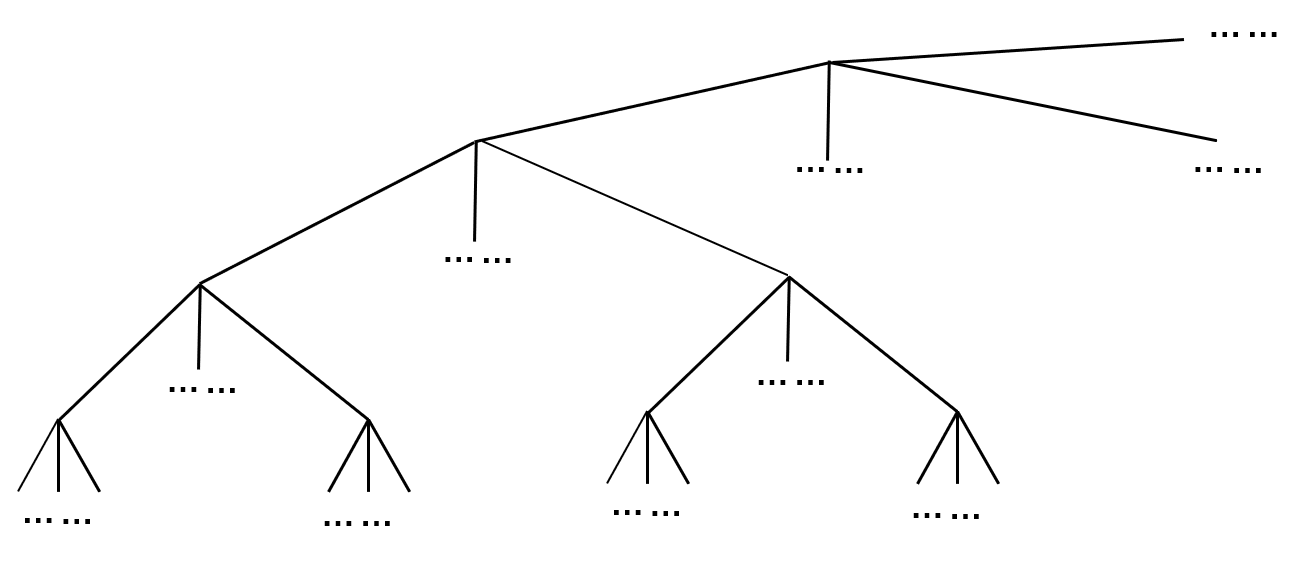}\\
	\caption{Tree model of $\Q_3$.}
	\label{Fig:tree model}
\end{figure}

\subsection{ Some $p$-adic integration theory}

With respect to addition, \(\Qp\) is a locally compact Abelian group, so it admits an additive Haar measure \(dx\),
a positive measure invariant under translation, i.e., \(d(x+a) = dx\) for all \(a \in \Qp\).
If we normalize this Haar measure by setting \(\int_{\Zp} dx = 1\), then \(dx\) is uniquely determined.

There is a key relationship between the additive and multiplicative structures of \(\Qp\), given by the following proposition.

\begin{proposition}[\cite{AKS}, Proposition 3.2.1]
For \( a \in \mathbb{Q}_p^*= \Qp \setminus \{0\}\), we have
$$d(xa) = |a|_p  dx.$$
\end{proposition}
Let \(\Qp^d\) denote the \(d\)-dimensional vector space over \(\Qp\). We equip \(\Qp^d\) with the maximum norm
\[
|{x}|_p = \max_{1 \leq j \leq d} \abspp{x_j} \quad \text{for } x = (x_1, \dots, x_d) \in \Qp^d.
\]
For \(x = (x_1, \dots, x_d)\) and \(y = (y_1, \dots, y_d)\) in \(\Qp^d\), the scalar product is defined as
\[
x \cdot y := x_1 y_1 + \dots + x_d y_d.
\]

The Haar measure on \(\Qp^d\) is the product measure \(dx_1 \cdots dx_d\), also denoted by \(dx\),
which is translation-invariant. For any linear isomorphism \(M: \Qp^d \to \Qp^d\) with \(\det M \neq 0\),
the measure satisfies
\[
 d(Mx) = \abspp{\det M} dx.
\]

For a measurable subset \(A \subset \Qp^d\) and \(\rho \geq 1\), let \(L^\rho(A)\) denote the space of
all measurable functions \(f: A \to \mathbb{C}\) such that
\[
\int_A \abs{f(x)}^\rho dx < \infty.
\]
The norm of \(f \in L^\rho(A)\) is defined as
\[
\norm{f}_{L^\rho} = \left(\int_A \abs{f(x)}^\rho dx\right)^{1/\rho}.
\]
In particular, when \(\rho = 2\), the space $L^2(A)$ is the Hilbert space of square Haar-integrable functions on $A$,

$$
L^2(A) =\Big\{f: A\rightarrow\mathbb{C}: \int_{A}|f(x)|^2dx<\infty \Big\},
$$
with the inner product
$$
\langle f, g\rangle_{A}=\int_{A} f(x)\overline{g(x)}dx, \quad \forall f, g \in L^2(A).
$$

\begin{lemma}[{\cite{AKS}, Proposition 3.2.2}]\label{dense}
 Let \(O \subset \Qp^d\) be an open set, and let \(C_c(O)\) denote the space of continuous functions on \(O\) with compact support. Then \(C_c(O)\) is dense in \(L^\rho(O)\) for all \(1 \leq \rho < \infty\).
\end{lemma}
Let \(O \subset \Qp^d\) be an open set. We define \(L_{\text{loc}}^\rho(O)\) as the space of functions \(f: O \to \mathbb{C}\) such that \(f \in L^\rho(K)\) for every compact subset \(K \subset O\). Clearly, \(L^\rho(O) \subset L_{\text{loc}}^\rho(O)\).

A function \(f\) is said to be integrable on \(\Qp^d\), denoted by \(f \in L^1(\Qp^d)\), if
\[
\int_{\Qp^d} |f(x)| dx<+\infty.
\]
Due to the topology of \(\Qp\), the integral of \(f\) over \(\Qp\) can be expressed as
\[
\int_{\Qp} f(x) dx = \lim_{N \to \infty} \sum_{y=-\infty}^N \int_{S_y} f(x) dx = \sum_{y=-\infty}^\infty \int_{S_y} f(x) dx.
\]

The Fubini's theorem and Lebesgue's dominated convergence theorem hold in \(\Qp^d\), as shown by the following lemmas.

 \begin{lemma}[{\cite{AKS}, Theorem 3.2.3}]\label{fubini}
 If a function \(f: \mathbb{Q}_p^{d_1 + d_2} \to \mathbb{C}\) satisfies that the repeated integral
\[
\int_{\mathbb{Q}_p^{d_1}} \left(\int_{\mathbb{Q}_p^{d_2}} f(x, y) dy\right) dx
\]
is finite, then \(f \in L^1(\mathbb{Q}_p^{d_1 + d_2})\), and the following equalities hold
\[\begin{split}
\int_{\mathbb{Q}_p^{d_1 + d_2}} f(x, y) dx dy
=&\int_{\mathbb{Q}_p^{d_1}} \left(\int_{\mathbb{Q}_p^{d_2}} f(x, y) dy\right) dx\\
=&\int_{\mathbb{Q}_p^{d_2}} \left(\int_{\mathbb{Q}_p^{d_1}} f(x, y) dx\right) dy.
\end{split}
\]

 \end{lemma}

 \begin{lemma}[{\cite{AKS}, Theorem 3.2.4}]\label{LDC}

Let \(\{f_k\}_{k \in \mathbb{N}} \subset L^1(\Qp^d)\) be a sequence of functions.
If \(\lim_{k \to \infty}f_k(x)=f(x)\) for almost every \(x \in \Qp^d\), and there exists
a function \(\psi \in L^1(\Qp^d)\) such that \(\abs{f_k(x)} \leq \psi(x)\) for all \(k \in \mathbb{N}\)
and almost every \(x \in \Qp^d\), then
\[
\lim_{k \to \infty} \int_{\Qp^d} f_k(x) dx = \int_{\Qp^d} f(x) dx.
\]
\end{lemma}

\subsection{Bruhat-Schwartz test functions in $\Qp^d$}

We provide a brief introduction to  Bruhat-Schwartz test functions in \(\Qp^d\),
mainly follows the literature \cite{AKS}.
\begin{definition}
A complex-valued function $\varphi$ defined on $\Q_p^d$ is called uniformly locally constant,
if \ $\exists \ n\in\mathbb{Z}$ such that
$$
  \varphi(x+y)=\varphi(x) \quad \forall x\in\Q_p^d,\  \forall y\in B_n^d.
$$
\end{definition}
Uniformly locally constant functions are continuous by definition. Let $\mathcal{U}(\Qp^d)$ denote the
space of all uniformly locally constant functions on $\Q_p^d$.
The space $\mathcal{D}(\Qp^d)$ of {\it Bruhat-Schwartz test functions} consists of uniformly locally
constant functions with compact support. Specifically,  any $\psi\in\mathcal{D}(\Qp^d)$ can be written
as a finite linear combination of indicator functions of the form
 $1_{B_\gamma^d(x)}(\cdot)$, where $\gamma\in\Z$, $x\in\Qp^d$, and
 \(B_\gamma^d(x) = B_\gamma(x) \times \dots \times B_\gamma(x) \subset \Qp^d\).
The largest \(\gamma\) for which such a representation holds is called the constancy
parameter of \(\psi\), denoted by \(\ell(\psi)\). Since \(\psi\) has compact support,
there exists a minimal \(m \in \mathbb{Z}\) such that \(\supp(\psi) \subset B_m^d\),
this \(m\) is called the compactness parameter of \(\psi\), denoted by \(m(\psi)\).
Clearly, $\mathcal{D}(\Qp^d)\subset\mathcal{U}(\Qp^d)$.
The space $\mathcal{D}(\Qp^d)$ is endowed with the topology as follows.
A sequence $\{\psi_n\}\subset\mathcal{D}(\Qp^d)$ is called a {\it null sequence} if
there is a fixed pair of $\gamma,\gamma'\in\Z$ such that: \begin{enumerate} \item
each $\psi_n $ is constant on every ball $B_\gamma^d(x)$ of radius $p^\gamma$; \item
each $\psi_n $ is supported by the ball $B_{\gamma'}^d(0)$;\item the sequence $\psi_n $ tends uniformly to zero.
\end{enumerate}
With this topology, the space $\mathcal{D}(\Qp^d)$ is a complete locally convex topological vector space.
The space of test functions $\mathcal{D}(\Qp^d)$ is dense in \(L^\rho(\Qp^d)\) for all \(1 \leq \rho < \infty\).

\subsection{Fourier transform and  spectral set criterion}

We first fix a character $\chi \in \widehat{\Q}_p$ defined by
$$\chi(x)=e^{2\pi i\{x\}},$$
where \(\{x\} = \sum_{n=v}^{-1} a_n p^n\) is the fractional part of \(x\). All characters of \(\Qp\)
are obtained by scaling this fixed character, for any \(y \in \Qp\), define \[\chi_y(x) = \chi(yx).\]
The map \(y \mapsto \chi_y\) is an isomorphism from \(\Qp\) to \(\widehat{\Qp}\).
Note that \(\chi(x) \equiv 1\) for all \(x \in \Zp\), but \(\chi\) is non-constant on \(p^{-1}\Zp\).

The dual group \(\widehat{\Qp^d}\) of \(\Qp^d\) consists of all characters \(\chi_y(\cdot)\) with \(y \in \Qp^d\),
where \(\chi_y(x) = \chi(x \cdot y) = e^{2\pi i \{x \cdot y\}}\).

For \(f \in L^1(\Qp^d)\), the Fourier transform \(\widehat{f}\) is defined as
$$
\widehat{f}(\xi) =\int_{\Q_p^d}f(x) \overline{\chi_{\xi}(x)} dx, \quad ( \xi\in \Q_p^d).
$$

Equivalently, this can be rewritten using \(\chi(-\xi \cdot x) = \overline{\chi(\xi \cdot x)}\),
\[
\widehat{f}(\xi) = \int_{\Qp^d} f(x) \chi(-\xi \cdot x) dx.
\]

The unit ball \(\Zp\) is an additive subgroup of \(\Qp\). Let \(\mathbb{L} \subset \Qp\) be a complete set of
representatives for the cosets of \(\Zp\) in \(\Qp\). Then
\[
\Qp = \mathbb{L} + \Zp = \bigsqcup_{\ell \in \mathbb{L}} (\ell + \Zp),
\]
where \(\bigsqcup\) denotes  disjoint union. This means \((\Zp, \mathbb{L})\) is a tiling pair,
\(\mathbb{L}\) is called a standard quasi-lattice in \(\Qp\), and \(\Zp\) is
the fundamental domain for \(\mathbb{L}\).

For any \(\lambda \in \Qp\), the restriction \(\chi_\lambda|_{\Zp}\) is a character on \(\Zp\).
Two such restrictions coincide if and only if \(\lambda - \lambda' \in \Zp\), that is,
\(\chi_\lambda = \chi_{\lambda'}\) if \(\lambda + \Zp = \lambda' + \Zp\), and
\(\chi_\lambda \neq \chi_{\lambda'}\) if \((\lambda + \Zp) \cap (\lambda' + \Zp) = \emptyset\).
Thus, if \(\mathbb{L}\) is a complete set of distinct coset representatives of \(\Zp\) in \(\Qp\),
the family \(\{\chi_\ell\}_{\ell \in \mathbb{L}}\) is a complete set of distinct characters on \(\Zp\).
This leads to the following lemma.

\begin{lemma}[\cite{Taibleson1975}]\label{lemlattice}
Let $\mathbb{L}$ be a complete set of distinct coset representatives of $\Z_p$ in $\Q_p$.
Then $(\Z_p,\mathbb{L})$ is a spectral pair.
\end{lemma}
Let \(|\Omega|\) denote the Haar measure of a set \(\Omega \subset \Qp^d\).
The following lemma provides a criterion for a Borel set \(\Omega \subset \Qp^d\)
with \(0 < |\Omega| < \infty\) to be a spectral set.
 \begin{lemma}[\cite{Fan}\label{orthogonal}]\label{criterion}
 Let $\Omega\subset\Q_p^d$ be a Borel set with $0<|\Omega|<\infty$. Then
 $(\Omega, \Lambda)$ is a spectral pair  if and only if
  \begin{equation}
 \sum_{\lambda\in\Lambda}|\widehat{1}_{\Omega}(\xi-\lambda)|^2=|\Omega|^2, \quad \forall \xi\in\Q_p^d.
  \end{equation}
 \end{lemma}

Define the family of exponential functions associated with a discrete set $\Lambda\subset\Q_p^d$ as

$$
\mathcal{E}(\Lambda)=\Big\{\chi_\lambda(x)=e^{2\pi i\{\lambda\cdot x\}}: \lambda\in\Lambda\Big\}.
$$
We say that $(\Omega, \Lambda)$ is an orthogonal pair,
if the system  $\mathcal{E}(\Lambda)$ forms an orthogonal system in $L^2(\Omega)$.

For a function $f:\Q_p^d\rightarrow \mathbb{C}$, we write
$$
    \mathcal{Z}(f)=\Big\{x\in\Q_p^d: f(x)=0\Big\}.
$$
The following lemma characterizes orthogonal pairs.
\begin{lemma}[\cite{K2025}\label{orthogonal}]
 If $\Omega\subset\Q_p^d$ is a Borel set with $0<|\Omega|<\infty$, then $(\Omega, \Lambda)$ is an orthogonal pair
 if and only if
  \begin{equation}\label{2.3}
 \Lambda-\Lambda\subset \mathcal{Z}(\widehat{1}_{\Omega})\cup\{0\}.
  \end{equation}
 \end{lemma}

The following lemma shows that both spectra and tiling complements of a bounded measurable
set of positive measure are necessarily uniformly discrete.

\begin{lemma}\label{discretespectra}
 Let $\Omega\subset\mathbb{Q}_{p}^{d}$ be a Borel set with $0<|\Omega|<\infty$.
\begin{enumerate}
\item If $(\Omega,\Lambda)$ is a spectral pair, then $\Lambda$ is uniformly discrete.
\item If $(\Omega,T)$ is a tiling pair, then $T$ is uniformly discrete.
\end{enumerate}
 \end{lemma}

\begin{proof}
(1) By the fact $\widehat{1}_{\Omega}(0)=|\Omega|>0$ and
the continuity of the function $\widehat{1}_{\Omega}(x)$, there exists a $n_0\in \mathbb{Z}$, such that
$\widehat{1}_{\Omega}(x)\neq0$ for all $x\in B_{n_0}^d$. This, together Lemma \ref{orthogonal}, implies
that $|\lambda-\lambda'|_p\geq p^{n_0}$ for different $\lambda,\lambda'\in\Lambda$,
which means $\Lambda$ is uniformly discrete with separation constant $\delta(\Lambda)\geq\delta_\Omega$,
where
\[
\delta_\Omega:=\min\Big\{|\xi|_p:  \xi\in \Qp^d, \quad \widehat{1}_\Omega(\xi)=0\Big\}.
\]

(2) Assume $(\Omega,T)$ is a tiling pair. By definition,  we have
\[
\sum_{t \in T} 1_{\Omega}(x - t) = 1\ \text{for} \ \text{a.e. } x \in \Qp^d.
\]
For distinct $t,t'\in T$, the translates $\Omega+t$ and $\Omega+t'$ are disjoint up to measure zero, so
\[
|\Omega\cap(\Omega+(t'-t))|=0.
\]
Define the function
\[
F(x):=|\Omega\cap(\Omega+x)| = \int_{\mathbb{Q}_p^d} 1_\Omega(y)1_\Omega(y+x)\,dy =
(1_\Omega * \widetilde{1}_\Omega)(x),
\]
where $\widetilde{1}_\Omega(y)=1_\Omega(-y)$. Then $F$ is continuous, and $F(0)=|\Omega|>0$.
Hence there exists $r_0\in\mathbb{Z}$ such that $F(x)>0$ for all $x\in B_{r_0}^d$.
For distinct $t,t'\in T$, the disjointness gives $F(t'-t)=0$, so $t'-t\notin B_{r_0}^d$.
Thus $|t-t'|_p \ge p^{r_0}$ for all distinct $t,t'\in T$,
proving that $T$ is uniformly discrete with separation constant $\delta(T)\geq\sigma_\Omega$,
where

\[
\sigma_\Omega:=\min\Big\{|x|_p:  x\in \Qp^d,\quad |\Omega\cap(\Omega+x)| =0\Big\}.
\]

\end{proof}

\section{Proof of main results}\label{main results}
This section proves the paper's three core theorems, the invariance of
functional tiles under weak convergence, the spectrality of product sets, and the equivalence between the
spectrality of cylindric sets and their lower-dimensional factors.

\subsection{Functional tiles}

Let $f\in L^1(\Q_p^d)$ be a non-negative function, and  $\Lambda\subset \Q_p^d$ be a discrete set.
Recall that $f$ tiles $\Q_p^d$ by translation along $\Lambda\subset\Q_p^d$, if

$$
\sum_{\lambda\in\Lambda}f(x-\lambda)=1, \quad a.e. \ x\in \Q_p^d.
$$

The following lemma establishes a critical link between the spectrality of a measurable
set $\Omega$ and the tiling property of a specific function derived from $\Omega$.

\begin{lemma}\label{spectral}
Let $\Omega\subset\Q_p^d$ be a bounded measurable set with positive measure, and define
$$
     f(x):=|\widehat{1}_{\Omega}(x)|^2/|\Omega|^2,\quad x\in \Q_p^d.
$$
Then $(\Omega, \Lambda)$  is a spectral pair if and only if $(f, \Lambda)$  is a tiling pair.
\end{lemma}
\begin{proof}
By Lemma \ref{criterion}, $(\Omega,\Lambda)$ forms a spectral pair if and only if
\begin{equation}\label{tiling}
\sum_{\lambda\in\Lambda}f(x-\lambda)=\sum_{\lambda\in\Lambda}|\widehat{1}_{\Omega}(x-\lambda)|^2/|\Omega|^2=1, \quad \forall x\in\Q_p^d.
\end{equation}
The equation (\ref{tiling}) holds if and only if $(f, \Lambda)$ is a tiling pair.
\end{proof}

The proof below leverages properties of Bruhat-Schwartz test functions and
weak convergence of sequences.

\begin{proof}[Proof of Theorem \ref{weak convergence}]

We aim to show that
\[
\sum_{\lambda \in \Lambda} f(x - \lambda) = 1 \quad \text{a.e.} \ x \in \Q_p^d.
\]
Let $\psi \in \mathcal{D}(\Qp^d)$ be a positive Bruhat-Schwartz test function  such that $\int_{\Qp^d} \psi(x) \, dx = 1$.
Let $\gamma_0$ denote the constancy parameter of $\psi$, i.e., $\psi$ is constant on each ball of radius $p^{\gamma_0}$, and let $m_0$ denote the compactness parameter, i.e., $\supp \psi \subset B_{m_0}^d$.

For each $n\in\mathbb{N}$, define
\[
\Psi_n(x) = \sum_{\lambda \in \Lambda_n} \psi(x + \lambda), \quad \Psi(x) = \sum_{\lambda \in \Lambda} \psi(x + \lambda).
\]

\textbf{Claim:} 
For every $\phi\in L^1(\mathbb{Q}_p^d)$, we have
\[
\lim_{n\to\infty}\int_{\mathbb{Q}_p^d}\phi(x)\Psi_n(x)\,dx
=
\int_{\mathbb{Q}_p^d}\phi(x)\Psi(x)\,dx.
\]

\textit{Proof of the claim.}
Fix $x\in \Qp^d$. We will show that $\Psi_n(x)\to \Psi(x)$ pointwise.

Since $\supp \psi \subset B_{m_0}^d$, we have $\psi(x+\lambda)\neq 0$ only if
 $\lambda\in B_{m_0}^d(-x)$. Let $\gamma(x)\in \mathbb{Z}$ be chosen such that
\[
B_{m_0}^d(-x)\subset B_{\gamma(x)}^d.
\]

Indeed, it suffices to take $p^{\gamma(x)}\ge \max\{p^{m_0}, |x|_p\}$,
then $B_{m_0}^d(-x)\subset B_{\gamma(x)}^d$.

Choose $\eta_1\in\mathbb{N}$ such that $p^{-\eta_1}<p^{\gamma_0}$,
so that $\psi$ is constant on balls of radius $p^{-\eta_1}$. By the weak convergence of
$\Lambda_n$ to $\Lambda$, there exists $N(x)\in\mathbb{N}$ (depending on $x$ through $\gamma(x)$)
such that for all $n\ge N(x)$,
\[
\Lambda_n\cap B_{\gamma(x)}^d \subset \Lambda + B_{-\eta_1}^d,\qquad
\Lambda\cap B_{\gamma(x)}^d \subset \Lambda_n + B_{-\eta_1}^d.
\]

We now show that each $\lambda\in \Lambda_n\cap B_{m_0}^d(-x)$ is matched to a unique
$\lambda_*\in \Lambda\cap B_{\gamma(x)}^d $ such that $|\lambda-\lambda_*|_p\le p^{-\eta_1}$.

Indeed, if one $\lambda\in\Lambda_n$ were matched to
two distinct points  $\lambda_*,\lambda_*'\in\Lambda$ both satisfy
$$|\lambda-\lambda_*|_p\le p^{-\eta_1},\quad |\lambda-\lambda_*'|_p\le p^{-\eta_1}.$$
Then, ultrametric inequality yields
\[
|\lambda_*-\lambda_*'|_p \le \max\big\{|\lambda_*-\lambda|_p,\ |\lambda-\lambda_*'|_p\big\} \le p^{-\eta_1}.
\]
Since $\Lambda$ is uniformly discrete with  separation constant $\delta(\Lambda)\ge\delta>0$,
so $|\lambda_*-\lambda_*'|_p\ge\delta$. But, for sufficiently large $\eta_1$, we have
$p^{-\eta_1}<\delta$, a contradiction. Therefore each $\lambda\in \Lambda_n\cap B_{m_0}^d(-x)$ maps to exactly one
$\lambda_*\in \Lambda\cap B_{\gamma(x)}^d$.  The symmetric inclusion ensures the reverse matching is also valid.

Hence, for all $n\ge N(x)$,
\[
\Psi_n(x)=\sum_{\lambda\in\Lambda_n}\psi(x+\lambda)
        =\sum_{\lambda_*\in\Lambda}\psi(x+\lambda_*)=\Psi(x).
\]

Therefore,

\begin{equation}\label{limit}
\lim_{n\to\infty}\Psi_n(x)=\Psi(x)\qquad \forall x\in\mathbb Q_p^d.
\end{equation}

Next,  we establish a uniform bound on  $\Psi_n$. Since $\Lambda_n$ is uniformly discrete
with separation constant $\delta(\Lambda_n)\ge \delta>0$, choose an integer $r\in\mathbb{Z}$ such that
\[
p^r<\delta.
\]
Then the balls $\{B_r^d(\lambda): \lambda\in\Lambda_n\}$ are pairwise disjoint. Indeed,
if $B_r^d(\lambda)\cap B_r^d(\lambda')\neq\emptyset$ for $\lambda\neq\lambda'$, then by the
ultrametric property, one ball contains the other,  since they have the same radius, they coincide,
implying $|\lambda-\lambda'|_p\le p^r<\delta$, contradicting the separation of $\Lambda_n$.

Now fix $x\in \Qp^d$. The condition $\psi(x+\lambda)\neq 0$ implies $\lambda\in B_{m_0}^d(-x)$.
The number of such $\lambda$ is bounded by the maximal number of disjoint balls of radius $p^r$ that
can fit inside $B_{m_0}^d(-x)$. Since $B_{m_0}^d(-x)$ has Haar measure $p^{d m_0}$ and each $B_r^d(\lambda)$
has Haar measure $p^{d r}$, we obtain
\[
\#(\Lambda_n\cap B_{m_0}^d(-x)) \le p^{d m_0}\cdot p^{-d r}=p^{d(m_0-r)}=:C,
\]
a constant independent of $n$ and $x$. Since $|\psi|\le \|\psi\|_\infty$,
therefore

\begin{equation}\label{uniforml}
\sup_{n}\|\Psi_n\|_\infty\le C\|\psi\|_\infty=: M,\qquad \|\Psi\|_\infty\le  C\|\psi\|_\infty=:M.
\end{equation}
uniformly over all $n\in\mathbb{N}$ and $x\in\Qp^d$.

From (\ref{limit}), for every $\phi\in L^1(\Qp^d)$,
we obtain  $$\lim_{n\to\infty}\phi(x)\Psi_n(x)=\phi(x)\Psi(x), \ \forall x\in\mathbb Q_p^d.$$

From inequality (\ref{uniforml}), we have
$$|\phi(x)\Psi_n(x)|\leq M|\phi(x)|\in L^1(\Qp^d).$$ So,
by Lemma \ref{LDC}, the dominated convergence theorem, we conclude
\[
\lim_{n\to\infty}\int_{\mathbb Q_p^d} \phi(x)\Psi_n(x)\,dx
= \int_{\mathbb Q_p^d} \phi(x)\Psi(x)\,dx.
\]

Now, since $(f, \Lambda_n)$ is a tiling pair, we have
\[
\sum_{\lambda \in \Lambda_n} f(x - \lambda) = 1 \quad \text{a.e.}, \forall  x \in \Q_p^d.
\]
Multiplying both sides by $\psi(x)$ and integrating, we get
\[
\int_{\Qp^d} \psi(x) \sum_{\lambda \in \Lambda_n} f(x - \lambda) \, dx = \int_{\Qp^d} \psi(x) \, dx = 1.
\]
Change variable $x \mapsto x + \lambda$ and apply Lemma \ref{fubini}, we obtain
\[
\int_{\Qp^d} f(x) \Psi_n(x) \, dx = 1.
\]

From the {\bf claim}, we get
\[
\lim_{n \to \infty} \int_{\Qp^d} f(x) \Psi_n(x) \, dx = \int_{\Qp^d} f(x) \Psi(x) \, dx.
\]
Hence,
\[
\int_{\Qp^d} f(x) \Psi(x) \, dx = 1.
\]
But note that
\[
\int_{\Qp^d} f(x) \Psi(x) \, dx =
\int_{\Qp^d} \psi(x) \sum_{\lambda \in \Lambda} f(x - \lambda) \, dx.
\]
Thus
\[
\int_{\Qp^d} \psi(x) \Big(\sum_{\lambda \in \Lambda}f(x - \lambda)-1\Big) \, dx = 0.
\]

Define
\[
F(x)=\sum_{\lambda\in\Lambda}f(x-\lambda).
\]
Since $\Lambda$ is uniformly discrete, for any compact open ball $B_\gamma^d$,
the set $$\Big\{\lambda\in\Lambda: B_\gamma^d\cap (B_\gamma^d+\lambda)\neq\emptyset\Big\}$$ is finite.
Furthermore, $f\geq0$ and $f\in L^1(\Qp^d)$,
hence
\[
\sum_{\lambda\in\Lambda}\int_{B_\gamma^d} f(x-\lambda)dx = \sum_{\lambda\in\Lambda}\int_{B_\gamma^d-\lambda}f(y)dy < \infty.
\]
Thus $F\in L^1_{\text{loc}}(\Qp^d)$.

Now, the identity above holds for every nonnegative Bruhat-Schwartz test function
$\psi\in\mathcal{D}(\Qp^d)$ with  $\int_{\Qp^d}\psi(x)\,dx=1$.
In particular, taking $\psi=1_{B_\gamma^d}$ (up to a constant factor) gives
\[
\int_{B_\gamma^d} \big(F(x)-1\big)dx = 0.
\]
for every compact open ball $B_\gamma^d$.
Since compact open balls generate the Borel $\sigma$-algebra, and Haar measure is a regular
Borel measure on $\Qp^d$, a locally integrable function whose integral over every compact open ball
vanishes must equal zero almost everywhere. Hence
\[
F(x)-1=0,\quad \text{a.e. }x\in\Qp^d.
\]
Thus,

\[
\sum_{\lambda \in \Lambda} f(x - \lambda) = 1 \quad \text{a.e.}\  x\in \Qp^d.
\]
Therefore, $(f, \Lambda)$ is a tiling pair.
\end{proof}

By using Lemmas \ref{discretespectra}, \ref{spectral} and Theorem \ref{weak convergence}, we obtain the following corollary.

\begin{corollary}\label{coro-weak limit}
Let $\Omega\subset\Q_p^d$ be a bounded, measurable set.\\
\indent {\rm (1)} \ Suppose that $(\Omega, \Lambda_n)$ is a spectral pair for each $n$.
If $\Lambda_n$ weakly converges to $\Lambda$, then $(\Omega, \Lambda)$ is also a spectral pair.\\
\indent {\rm (2)} \ Suppose that $(\Omega, \Lambda_n)$ is a tiling pair for each $n$. If $\Lambda_n$ weakly converges
to $\Lambda$, then $(\Omega, \Lambda)$is also a tiling pair.
\end{corollary}

\subsection{Proof of Theorem \ref{product}}

Recall
$$
\mathcal{E}(\Lambda)=\Big\{\chi_\lambda(x)=e^{2\pi i\{\lambda\cdot x\}}:
\ \lambda\in\Lambda\Big\},\quad \forall x\in\Q_p^d.
$$
$(\Leftarrow)$ Let (\(\Omega_1, \Lambda_1)\) and \((\Omega_2, \Lambda_2)\) be spectral pairs.
We first show that the set $\mathcal{E}(\Lambda)$ is orthogonal in $L^2(\Omega_1\times\Omega_2)$.

\[\begin{split}
\langle\chi_\lambda,\chi_{\lambda^{\prime}}\rangle_\Omega
=&\iint_\Omega\chi_\lambda(x,y)\overline{\chi_{\lambda^{\prime}}(x,y})dxdy\\
=&\int_{\Omega_1}\chi_{\lambda_1-\lambda_1^{\prime}}(x)(\int_{\Omega_2}\chi_{\lambda_2-\lambda_2^{\prime}}(y)dy)dx
\end{split}\]
If $\lambda_1\neq\lambda_1^{\prime}$ in $\Lambda_1$, then the inner product equals to zero since $(\Omega_1,\Lambda_1)$
is a spectral pair.
If $\lambda_1=\lambda_1^{\prime}$, but $\lambda_2\neq\lambda_2^{\prime}$, then it also equals to zero
since $(\Omega_2,\Lambda_2)$ is a spectral pair. This proves orthogonality of $\Lambda$.
In order to prove the completeness, what we have to prove is that for any $f\in L^2(\Omega)$
such that $\langle \chi_\lambda,f\rangle_\Omega=0,\quad \forall \lambda\in\Lambda$, then $f=0 \quad a.e.$

The inner product is

\[\begin{split}
\langle \chi_\lambda, f\rangle_\Omega
=&\iint_\Omega\chi_\lambda(x,y)\overline{f(x,y)}dxdy\\
=&\int_{\Omega_2}\chi_{\lambda_2}(y)\bigg(\int_{\Omega_1}\chi_{\lambda_1}(x)\overline{f(x,y)}dx\bigg)dy
\end{split}\]
If we fixed $\lambda_1$, and the double integral vanishes for all $\lambda_2\in\Lambda_2$,
then, since $(\Omega_2,\Lambda_2)$ is a spectral pair, the integral
$\int_{\Omega_1}\chi_{\lambda_1}(x)\overline{f(x,y)}dx=0$ for
almost all $y$. But $\lambda_1$ is arbitrary, so the spectral pair property of $(\Omega_1,\Lambda_1)$
implies that $f=0$ for almost all $x$. Thus, we proved that $\mathcal{E}(\Lambda)$ is
complete on $L^2(\Omega)$.

(\( \Rightarrow\)) Let \((\Omega_1\times\Omega_2,\ \Lambda_1\times\Lambda_2)\) be a spectral pair in $\Q_p^{d_1}\times\Q_p^{d_2}$.
We prove that \((\Omega_2, \Lambda_2)\) is a spectral pair in $\Q_p^{d_2}$ (the proof that \((\Omega_1, \Lambda_1)\) is a
spectral pair in $\Q_p^{d_1}$ is identical).
First, we check orthogonality of $\mathcal{E}(\Lambda_2)$ in $L^2(\Omega_2)$. Let $\alpha, \beta$ be two distinct points in $\Lambda_2$,
then the points $\lambda=(\lambda_1, \alpha)$ and $\lambda^{\prime}=(\lambda_1, \beta)$ are
two distinct points in $\Lambda$. Since $(\Omega_1\times\Omega_2, \Lambda_1\times\Lambda_2)$ is a spectral pair, the inner product
$$
\langle \chi_\lambda, \chi_{\lambda^{\prime}}\rangle_{\Omega_1\times\Omega_2}=0.
$$

On the other hand,
$$\langle \chi_\lambda, \chi_{\lambda^{\prime}}\rangle_{\Omega_1\times\Omega_2}
=|\Omega_1|\langle \chi_\alpha, \chi_\beta\rangle_{\Omega_2}.$$
Since $|\Omega_1|>0$, we must have
$$
\langle \chi_\alpha, \chi_\beta\rangle_{\Omega_2}=0,
$$
which implies the orthogonality of $\mathcal{E}(\Lambda_2)$ in $L^2(\Omega_2)$.

Now we shall prove that this system is also complete in $L^2(\Omega_2)$. Take any
$\varphi\in L^2(\Omega_2)$ such that
\begin{equation}\label{3.7}
\langle \chi_\alpha, \varphi\rangle_{\Omega_2}=0 \quad \forall \alpha\in \Lambda_2.
\end{equation}
Consider a function $\Phi$ defined on $\Omega=\Omega_1\times\Omega_2$ by
$$
\Phi(x,y):=\varphi(y)\quad (\text{constant in }x).
$$
Then $\Phi\in L^2(\Omega)$. We claim that $\Phi$ is orthogonal in $L^2(\Omega)$ to all the elements of
the system $\mathcal{E}(\Lambda)$.
Indeed, let  $\lambda\in\Lambda$ be any element of $\Lambda$ , then it has the form
$\lambda=(\lambda_1, \lambda_2)\in \Lambda_1\times\Lambda_2$. This implies that
\begin{equation}\label{3.8}
\langle \chi_\lambda, \Phi\rangle_{\Omega}= C_{\lambda_1}
\cdot\langle \chi_{\lambda_2}, \varphi\rangle_{\Omega_2},
\end{equation}
where $C_{\lambda_1}:=\int_{\Omega_1}\chi_{\lambda_1}(x) dx$ is a constant, due to (\ref{3.7}),
the inner product must vanish. This confirms that
$$
 \langle \chi_\lambda, \Phi\rangle_{\Omega}=0, \quad \forall \lambda\in\Lambda.
$$
By the completeness of the system $\mathcal{E}(\Lambda_1\times\Lambda_2)$ in $L^2(\Omega_1\times\Omega_2)$,
the function $\Phi$ vanishes identically almost everywhere, so $\varphi=0$ a.e. This implies the desired results.

\subsection{Cylindric sets in $\Q_p^d$}

In this section, we assume that $\Omega$ is a cylindric set in $\Q_p^d$, namely
\begin{equation}\label{cylindric set}
    \Omega=B_{\gamma}(a)\times\Theta.
\end{equation}
where $B_{\gamma}(a)\subset\Q_p$ is a ball, and $\Theta$ is a bounded measurable set in $\Q_p^{d-1}$.
Since the spectrality of a set is invariant under dilations and translations, we can assume that
$B_{\gamma}(a)=\Z_p$ for simplicity. First, we give a very useful lemma.

\begin{lemma}\label{ballspectrum}
We have $\widehat{1}_{\Zp}(\xi)=\left\{ \begin{array}{rcl}
1,& \mbox{if}& |\xi|_p\leq1; \\
0 ,& \mbox{if} & |\xi|_p>1.
\end{array}\right.$
 \end{lemma}

 \begin{proof} Recall that
 $$
    \widehat{1}_{\Zp}(\xi)
    = \int_{\Zp}e^{-2\pi i\{\xi x\}}dx.
 $$
 When $|\xi|_p\le 1$, the integrand is equal to $1$, so $\widehat{1}_{\Zp}(\xi)=1$.
 When $|\xi|_p>1$, making a change of variable $x = y-z$ with $z\in\Zp$ chosen
 such that $ e^{2\pi i\{\xi z\}}\not=1$,
 we get $$\widehat{1}_{\Zp}(\xi) = e^{2\pi i\{\xi z\}} \widehat{1}_{\Zp}(\xi).$$
 It follows that $\widehat{1}_{\Zp}(\xi)=0$ for $|\xi|_p>1$, which means
 $\widehat{1}_{\Zp}(\xi)=0$ for $\xi\notin\Zp$.
 \end{proof}

We will denote a point $x\in\Q_p^d$ as $x=(x_1, x_2)\in \Q_p\times\Q_p^{d-1}$.
Due to (\ref{cylindric set}), the Fourier transform $\widehat{1}_{\Omega}$ of the indicator
function $1_{\Omega}$ of the set
 $\Omega$ is given by
\begin{equation}
\widehat{1}_{\Omega}(\xi)=\widehat{1}_{\Z_p}(\xi_1)\widehat{1}_{\Theta}(\xi_2),
\quad \xi=(\xi_1, \xi_2)\in\Q_p\times\Q_p^{d-1} .
\end{equation}

By Lemma \ref{ballspectrum}, we have
$$\mathcal{Z}(\widehat{1}_{\Z_p})=\Qp\setminus \Zp,$$
so one has
$$\mathbb{L}\setminus\{0\}\subset\mathcal{Z}(\widehat{1}_{\Zp}).$$
Then it is easy to prove the following.

\begin{lemma}\label{lattice point}
 Let $\lambda,\lambda^{\prime}$ be two distinct points in $\Q_p^d$. The  exponential
 functions $\chi_{\lambda}$ and $\chi_{\lambda^{\prime}}$
 are mutually orthogonal in $L^2(\Omega)$ if and only if
 \[
 \lambda_1-\lambda^{\prime}_1\notin\Zp\ \text{or}\
\lambda_2-\lambda^{\prime}_2\in\mathcal{Z}(\widehat{1}_{\Theta}).
\]
\end{lemma}

\begin{lemma}\label{lemintegerspctrum}
 Assume that $\Omega= \Zp\times\Theta$ is a spectral set, where $\Theta\subset\Q_p^{d-1}$ is bounded, measurable set.
 Then $\Omega$ admits a spectrum $\Gamma$ satisfying
\begin{equation}\label{integer spectrum}
              \Gamma\subset\mathbb{L}\times\Q_p^{d-1}.
\end{equation}
\end{lemma}

\begin{proof}

Since $\Omega$ is a spectral set in $ \Qp^d$, it has a spectrum $\Lambda \subset \Qp^d$.
For each $\lambda = (\lambda_1, \lambda_2) \in \Lambda$, where $\lambda_1\in \Qp, \ \lambda_2\in\Qp^{d-1}$,
by definition of $\mathbb{L}$, there exists a unique $\ell_\lambda \in \mathbb{L}$ with $\lambda_1 - \ell_\lambda \in \Zp$.
Let $z_\lambda = \lambda_1 - \ell_\lambda \in \Zp$, so $\lambda = (\ell_\lambda + z_\lambda, \lambda_2)$.
Define $$\gamma_\lambda = (\ell_\lambda, \lambda_2) \in \mathbb{L} \times \Qp^{d-1}.$$
Let $\Gamma = \{\gamma_\lambda \mid \lambda \in \Lambda\}$, and we show that $\Gamma$ is a spectrum of $\Omega$.

 For any distinct $\gamma_\lambda, \gamma_{\lambda'} \in \Gamma$,
if $\ell_\lambda \neq \ell_{\lambda'}$,
since $\ell_\lambda, \ell_{\lambda'}$ are distinct coset representatives, then, we have that  $\ell_\lambda - \ell_{\lambda'} \notin \Zp$. By Lemma \ref{ballspectrum}, we have
\[
 \ell_\lambda - \ell_{\lambda'} \in \mathcal{Z}(\widehat{1}_{\Zp}).
\]
Thus, by Lemma \ref{lattice point}, $\chi_{\gamma_\lambda}$ and $ \chi_{\gamma_{\lambda'}}$ are mutually orthogonal in $L^2(\Omega)$.

If $\ell_\lambda = \ell_{\lambda'}$,
since $\gamma_\lambda \neq \gamma_{\lambda'}$, we have $\lambda_2 \neq \lambda'_2$. For $\lambda \neq\lambda' \in \Lambda$,
$\chi_\lambda$ and $\chi_{\lambda'}$ are mutually orthogonal in $L^2(\Omega)$. Their first coordinates satisfy
\[
\lambda_1 - \lambda_1' = (\ell_\lambda + z_\lambda) - (\ell_\lambda + z_{\lambda'}) = z_\lambda - z_{\lambda'} \in \Zp.
\]
Thus, $\lambda_1 - \lambda_1' \notin \mathcal{Z}(\widehat{1}_{\Zp})$. By Lemma \ref{lattice point}, orthogonality implies
\[
\lambda_2 - \lambda_2' \in \mathcal{Z}(\widehat{1}_{\Theta}).
\]
But $\gamma_{\lambda_2} - \gamma_{\lambda'_2} = \lambda_2 - \lambda_2'$, so, by Lemma \ref{lattice point},
 $\chi_{\gamma_\lambda}$ and $\chi_{\gamma_{\lambda'}}$ are mutually orthogonal in $L^2(\Omega)$.

In order to prove the completeness of $\{\chi_{\gamma_\lambda}: \gamma_\lambda\in \Gamma\}$ in $L^2(\Omega)$,
we have to show that, for any $f \in L^2(\Omega)$ satisfies $\langle \chi_{\gamma_\lambda}, f \rangle = 0$
for all $\gamma_\lambda \in \Gamma$, we have  $f = 0$ a.e.

For any $\lambda \in \Lambda$, note $\lambda = (\ell_\lambda + z_\lambda, \lambda_2)$ with $z_\lambda \in \Zp$. For $(x, y) \in \Zp \times \Theta$, we have $\{{z_\lambda x}\} = 0$.
Thus
\[
\{{\lambda \cdot (x, y)}\} = \{{(\ell_\lambda + z_\lambda)x + \lambda_2 y}\} = \{{\ell_\lambda x + \lambda_2 y}\} = \{{\gamma_\lambda \cdot (x, y)}\}
\]
Hence, $\chi_\lambda = \chi_{\gamma_\lambda}$ on $\Omega$, so $\langle \chi_\lambda, f \rangle_\Omega
= \langle \chi_{\gamma_\lambda}, f\rangle_\Omega=0$.
Since $\mathcal{E}(\Lambda)$ is complete in $L^2(\Omega)$, we have that $f=0$ almost everywhere.
Thus $\Gamma \subset \mathbb{L} \times \Qp^{d-1}$ is a spectrum for $\Omega$.
\end{proof}

The following lemma is a consequence of the orthogonal decomposition.
\begin{lemma}\label{lembasespectrum}
 Suppose that $(\Omega,\Lambda)$ is a spectral pair with $\Lambda\subset\mathbb{L}\times\Q_p^{d-1}$.
Then for each $\ell\in\mathbb{L}$, the fiber
\begin{equation}\label{equ-base spectrum}
\Gamma_\ell:=\Big\{\gamma\in\Q_p^{d-1}:(\ell, \gamma)\in\Lambda \Big\},
\end{equation}
constitutes a spectrum for $\Theta$.
\end{lemma}

\begin{proof}
We make use of the orthogonal decomposition
\[
L^2(\mathbb{Z}_p\times\Theta)=\bigoplus_{\ell\in\mathbb{L}}\chi_\ell \cdot L^2(\Theta),
\]
where
 \[
\chi_\ell \cdot L^2(\Theta)=\{f\in L^2(\mathbb{Z}_p\times\Theta): f(x,y)=\chi_\ell(x)g(y)\text{ for some }g\in L^2(\Theta)\}.
\]
This decomposition follows from the fact that $\{\chi_\ell:\ell\in\mathbb{L}\}$ forms an orthonormal basis for $L^2(\mathbb{Z}_p)$.

\textbf{Orthogonality of $\mathcal{E}(\Gamma_\ell)$.} For distinct $\gamma,\gamma'\in\Gamma_\ell$,
the points $(\ell,\gamma)$ and $(\ell,\gamma')$ are distinct elements of $\Lambda$. Since $(\Omega,\Lambda)$
is a spectral pair, the corresponding exponentials are orthogonal in $L^2(\Omega)$
\[
0=\langle \chi_{(\ell,\gamma)},\chi_{(\ell,\gamma')}\rangle_\Omega
=|\mathbb{Z}_p|\,\langle \chi_\gamma,\chi_{\gamma'}\rangle_\Theta.
\]
Thus $\mathcal{E}(\Gamma_\ell)$ is orthogonal in $L^2(\Theta)$.

\textbf{Completeness of $\mathcal{E}(\Gamma_\ell)$.} Let $g\in L^2(\Theta)$ be such that
\[
\langle \chi_\gamma,g\rangle_\Theta=0\quad \forall \gamma\in\Gamma_\ell.
\]
Define $f\in L^2(\mathbb{Z}_p\times\Theta)$ by
\[
f(x,y)=\chi_\ell(x)g(y).
\]
We have
\[
\left\langle
\chi_{(\ell',\gamma')},f
\right\rangle_{(\mathbb Z_p\times\Theta)}
=
\left\langle
\chi_{\ell'},\chi_\ell
\right\rangle_{\mathbb Z_p}
\left\langle
\chi_{\gamma'},g
\right\rangle_{\Theta}.
\]
Let $\lambda=(\ell',\gamma')\in\Lambda$ be any element of $\Lambda$ satisfing \(\ell'\ne\ell\),
then $\ell-\ell'\notin\mathbb{Z}_p$, by Lemma~\ref{lattice point}, the first factor equals to zero.

When \(\ell'=\ell\), the second factor vanishes by the assumed orthogonality of \(g\) to \(\mathcal{E}(\Gamma_\ell)\).
These two cases establish that \(f\) is orthogonal to the entire spectral basis,
and completeness of $\mathcal{E}(\Lambda)$ in $L^2(\Omega)$ gives \(f=0\) a.e., hence \(g=0\) a.e.
Thus $\mathcal{E}(\Gamma_\ell)$ is complete in $L^2(\Theta)$. Therefore, $\Gamma_\ell$ is a spectrum for $\Theta$.
\end{proof}

Proof of Theorem \ref{main} now follows easily from the previous lemmas.

\begin{proof}[Proof of Theorem \ref{main}]

Let $\Omega=B_{\gamma}(a)\times\Theta$ be a cylindric set in $\Q_p^d$.
Without loss of generality, we can assume that $B_{\gamma}(a)=\Z_p$.

If $\Omega$ is  spectral, Lemma \ref{lemintegerspctrum} guarantees existence of a spectrum  $\Lambda\subset\mathbb{L}\times\Q_p^{d-1}$.
Then Lemma \ref{lembasespectrum} implies that $\Theta$ is a spectral set. Conversely, assume that $\Theta$ is a spectral set with
$\Gamma\subset\Q_p^{d-1}$ is a spectrum. Then Lemma \ref{lemlattice} and Theorem \ref{product} implies that $\Lambda=\mathbb{L}\times\Gamma$
is a spectrum for $\Omega$, and hence $\Omega$ is a spectral set.
\end{proof}

\textbf{Remark}.
The implication
 $$\Omega_1\times\Omega_2\ \mbox{spectral} \Rightarrow \Omega_1 \ \text{spectral and}\ \Omega_2 \ \text{spectral}$$
is very important for the Fuglede's conjecture. If we suppose that the `` spectral $\Rightarrow$
tile " half of the  conjecture is true in $\Q_p^{d_1+d_2}$ and the `` tile $\Rightarrow$ spectral "
half to be true in $\Q_p^{d_1}$ and in $\Q_p^{d_2}$, then it follows that if
$\Omega_1\times\Omega_2\subset\Q_p^{d_1}\times\Q_p^{d_2}$  is spectral
then so are $\Omega_1\subset\Q_p^{d_1}$ and $\Omega_2\subset\Q_p^{d_2}$.
In particular, let $\Omega_1,\Omega_2\subset\Q_p$ be compact open sets which are $p$-homogeneous and $\Omega=\Omega_1\times\Omega_2\subset\Q_p^2$.
If we can prove that``$\Omega$ is spectral $\Rightarrow$ $\Omega$ is a tile" is true in $\Q_p^2$, then
combining our results with the Theorem $1.1$ in \cite{Fan-Fan-Shi}, we can conclude that the conjecture holds
for this kind of $\Omega\subset\Q_p^2$.

\textbf{Declarations}

\textbf{ Acknowledgements:} 
The author sincerely thanks the referee for the careful reading of the manuscript
and for many insightful comments and suggestions. In particular, the remarks concerning
Lemma \ref{lattice point}, Lemma \ref{lembasespectrum}, and the proof of Theorem \ref{weak convergence}
have substantially improved the quality of the paper.

\textbf{Conflict of interest:} The author has no relevant financial or non-financial interests to disclose.

\textbf{Data availability statement:} Data sharing is not applicable to this article as no data
sets were generated or analysed during the current study.

\textbf{Fundings:} This research is supported by NSF of China (Grant No. 12361015), and
by NSF of Xinjiang Uygur Autonomous Region, P. R. China (Grant No. 2025D01A09).


\end{document}